\documentclass[11pt]{amsart}

\usepackage{amssymb, amsthm, amsmath}
\newtheorem{theorem}{Theorem}[section]

\newtheorem{lemma}[theorem]{Lemma}
\newtheorem{corollary}[theorem]{Corollary}
\newtheorem{prop}[theorem]{Proposition}

\newtheorem*{theorem*}{Theorem}{\bf}{\it}
\newtheorem*{proposition*}{Proposition}{\bf}{\it}
\newtheorem*{observation*}{Observation}{\bf}{\it}
\newtheorem*{lemma*}{Lemma}{\bf}{\it}

\theoremstyle{definition}

\newtheorem{example}[theorem]{Example}

\theoremstyle{remark}
\newtheorem{remark}[theorem]{Remark}

\newcommand{\F}{\mathcal F}
\newcommand{\Z}{\mathbb Z}
\newcommand{\R}{\mathbb R}
\newcommand{\C}{\mathbb C}
\newcommand{\dv}{{\rm{div}}}
\renewcommand{\Re}{\mathrm{Re}}

\begin{document}
\title{On ratios of harmonic functions}
\author{Alexander Logunov}
\email{log239@yandex.ru}
\author{Eugenia Malinnikova}
\email{eugenia@math.ntnu.no}

\begin{abstract}
Let $u$ and $v$ be harmonic in $\Omega \subset \R^n$ functions with the same zero set $Z$.
We show that the ratio $f$ of such functions is always well-defined and is real analytic. Moreover it satisfies the maximum and minimum principles. For $n=3$ we also prove the Harnack inequality and the gradient estimate for the ratios of harmonic functions, namely ${  \sup\limits_{K} |f| \leq C \inf\limits_{K}| f| \quad \& \quad    \sup\limits_{K} |\nabla f| \leq C \inf\limits_{K}| f|  }$  for any compact subset $K$ of $\Omega$, where the constant $C$ depends  on $K$, $Z$, $\Omega$ only. In dimension two the first inequality follows from the boundary Harnack principle and the second  from  the gradient estimate recently obtained  by Mangoubi. It is an open question whether these inequalities remain true in higher dimensions ($n \geq 4$). 
\end{abstract}

\maketitle
\section{Introduction}

\subsection{Motivation and main results}
The classical Harnack principle and it's corollary claim that if $K$ is a compact subset of a domain $\Omega$, then there exists a  positive constant $C$ such that for any positive and harmonic in $\Omega$ function $u$  
$$\inf\limits_{K} u\geq C \sup\limits_{K} u \quad \& \quad \inf\limits_{K} u \geq C   \sup\limits_{K} \left|\nabla u\right|.$$
 
 In order to find a proper extension of this principle to harmonic functions changing a sign we consider the ratios of harmonic functions sharing the same zero set.
Let $u$ and $v$ be harmonic functions in $\Omega \subset \R^n$ that vanish at exactly the same set $Z\subset \Omega$, we call this set the nodal set of $u$ (and $v$) and write $Z=Z(u)$. We study the ratio $f=u/v$. 
For general real analytic functions having the same set of real zeros the ratio is not always well-defined, there are also examples when the ratio is a continuous but not differentiable function. The situation changes  when we assume that the functions are harmonic. Our first result, so called local division theorem, says that the ratio of two harmonic functions with common nodal set is real analytic. It implies that if real zeros of two harmonic functions coincide then their complex zeros also coincide in some neighborhood of the real plane.

 Our second result says that for $n= 3$ there exists $C=C(Z,K,\Omega)$ such that
\begin{equation}\label{H2}
(a)\ \inf\limits_{K}| f|\geq C \sup\limits_{K} |f| \quad \& \quad (b)\ \inf\limits_{K}| f|  \geq C   \sup\limits_{K} |\nabla f|,
\end{equation}
when $f=u/v$ and $u$ and $v$ are harmonic functions satisfying $Z(u)=Z(v)=Z$.
 If $Z$ is the empty set, the last statement follows from the classical Harnack principle.

 Ratios of positive harmonic functions frequently appear in classical potential theory, in particular in connection with the Martin boundary and the (boundary) Harnack principle (see for example \cite{AG},\cite{A}) and the Green function (3G inequalities, see \cite{CFZ}). 
Our interest in the ratios of harmonic functions changing a sign grew up from studying the recent work of Dan Mangoubi, \cite{M}. 
The following   result  was  obtained in \cite{M} in dimension two.
\begin{theorem*}[Mangoubi] Let $Z\subset B_2=\{x:|x|<2\}\subset \R^2$, denote 
\[
\F(Z)=\left\{u:B_2\rightarrow \R, \Delta u=0, Z(u)=Z\right\}.\]
Then for any $u,v\in\F(Z)$ the ratio $f=u/v$ extends to a smooth nowhere vanishing function in $B_2$ and there exists a constant $C_Z>0$ such that $|\nabla \log |f||\le C_Z$ in $B_1$.
\end{theorem*}
We refer the reader to \cite{M} for motivation of the problem, its connection to Li-Yau's gradient estimate, and a list of examples of harmonic functions sharing the zero set.

 A connected component $\Omega$ of $B\setminus Z$ is called a nodal domain.  In dimension two $\partial \Omega \cap B$ can be represented locally as a graph of a Lipschitz function, then  one can apply the boundary Harnack principle for Lipschitz domains (see for example \cite{AG}) to see that $u=fv$ for some locally bounded function $f$ that does not change the sign. In higher dimensions the geometry of the nodal domains can be much more complicated. We give an example illustrating  that already in dimension three the nodal domains may violate the Harnack chain condition. Thus there exists a harmonic function $v$ such that $B\setminus Z(v)$ has components that are not NTA domains (see \cite{JK} for the definition); this creates an obstacle for the direct application of the boundary Harnack principle, which would have implied the boundedness of $f$. However by the local division theorem the ratio $f$ is always defined and is real analytic.  Moreover our proof of the local division result shows that the Harnack principle for the ratios  (\ref{H2}a) implies the gradient estimate. Thus the main problem in generalizing Mangoubi's theorem to higher dimensions is in establishing (\ref{H2}a). We are able to prove this inequality only  in dimension three, for this case the structure of the critical set of a harmonic function (where the function and its gradient simultaneously vanish) is less complicated than in higher dimensions.  We prove the following result, which will be referred to as the Harnack inequality for the ratios of harmonic functions.

\begin{theorem} \label{th:H} Assume that $w$ is a harmonic function in the unit ball $B\subset\R^3$. For any compact subset $K\subset B$ there exists a constant $C$ that depends on $w$ and $K$ only such that for any harmonic functions $u,v$ in $B$ such that $Z(u)=Z(v)=Z(w)$ and any points $x,y\in K$ we have
\[
\left|\frac{u}{v}(x)\right|\le C\left|\frac{u}{v}(y)\right|.\] 
\end{theorem}  

 This result combined with the local division argument gives estimates for the derivatives of the ratios of harmonic functions. 
\begin{theorem}\label{th:full}
Suppose $u$ and $v$ are harmonic functions in $\Omega\subset \mathbb{R}^3$. If $Z(v) = Z(u)=Z$, then there exists a real analytic function $f$ such that $u=vf$. If we fix $x_0 \in \Omega$ and assume $\frac{u}{v}(x_0)=1$, then for any compact set $K\subset \Omega$ there exist  positive numbers $A$ and $R$ depending only on $K, Z$ and $\Omega$ such that for all $x \in K$ and any multi-index $\alpha$ 
$$\left|D^{\alpha} \left(\frac{u}{v}\right)(x)\right|\leq \alpha!A R^{|\alpha|}.$$
 \end{theorem}

The extension of Magoubi's estimate of $|\nabla \log |f||$ to dimension three immediately follows from Theorem \ref{th:H} and Theorem \ref{th:full}.


\subsection{Outline of the proof}
In order to study the ratios of harmonic functions we want to understand the local behavior of a harmonic function near its zero point. In higher dimensions the structure of the zero set of a harmonic function could be very complicated. However the following key observations still hold: (i) locally the zero set resembles that of a harmonic polynomial (at least in some sense, see Lemma \ref{l3} and also Counterexample \ref{e:NH}), (ii) if $P$ and $Q$ are homogeneous harmonic polynomials satisfying $Z(Q)\subset Z(P)$ then $Q|P$ as a polynomial. Our first step is division of harmonic polynomials with common set of zeros. The main tool is the Brelot-Choquet theorem which says that a non-constant factor of a harmonic polynomial changes sign, see \cite{BC}. The lemma on division that we need follows form the results of B.H.~Murdoch \cite{Mur};  some facts on divisibility similar to what we use can be also found in  \cite[Section 5]{MOV}, where division of harmonic polynomials is applied to estimates of the maximal singular operators.

The result on division of harmonic polynomials allows us to write the ratio of two harmonic functions with the same zero set $Z$ as a formal power series centered at any point of $Z$. Further, using some intrinsic estimates, we show that  this series has positive radius of convergence, hence the ratio is a real analytic function. Then we establish the maximum (and minimum) principle for the ratios of harmonic functions. 

Next step is to prove Theorem \ref{th:H}. The main idea is to combine  the Maximum Principle and the Boundary Harnack Principle.  In dimension three we prove the following structure lemma: there is a countable set $D$ with locally finitely many accumulation points such that for any neighborhood $V$ of $D$ near each point of  $Z\setminus V$ the boundaries of all nodal domains are graphs of Lipschitz functions. When the structure lemma is obtained, the rest of the argument is relatively simple. We choose a ball $B_r$  that contains $K$ and such that $S_r=\partial B_r$ does not contain any points of $D$ or any accumulation point of $D$. Applying the boundary Harnack principle for parts of nodal domains near $S_r$ we conclude that $\max_{S_r} |f|\le C\min_{S_r}|f|$. Then the maximum and minimum principles for $f$  give the required Harnack inequality.

\subsection{Acknowledgments} We are grateful  to Dan Mangoubi for explaining his work \cite{M} to one of us. The present work was mostly carried out when the first  author visited the Department of Mathematical Sciences of the Norwegian University of Science and Technology and the second author visited Chebyshev Laboratory at St.Petersburg State University. It is a pleasure to thank both institutions for their support and great working conditions. The first author was supported by the Chebyshev Laboratory (Department of Mathematics
and Mechanics, St. Petersburg State University) under the RF Government grant
11.G34.31.0026, and by JSC "Gazprom Neft"; the second author was supported by Project 213638 of the Research Council of Norway.


\section{Local division of power series of harmonic functions}  
In this section we discuss the local division of harmonic functions with the common set of zeros $Z$ in $\Omega\subset\R^n$. We show that for any $a\in \Omega$ there is a power series $f_a$ such that $u(x)=v(x)f_a(x)$ as formal power series centered at $a$. We take $a\in Z$ and assume that $a=0$ to simplify the notation. 

\if false

\fi

\subsection{Division of Harmonic Polynomials and Formal Power Series}
Let $P$ and $Q$ be polynomials in $\mathbb{R}[x_1, x_2, \dots, x_n]$. We are interested in conditions on $P$ and $Q$  ensuring the divisibility of  $P$ by $Q$. If $P$ is divisible by $Q$, then surely $Z(Q) \subset Z(P)$. The converse statement is false in general but it appears to be true if  $Q$ is a homogeneous harmonic polynomial. 

\begin{lemma}[Division Lemma] \label{dl2}
 Suppose $Q$ is a homogeneous harmonic polynomial and $P$ is a polynomial such that $Z(Q) \subset Z(P)$. Then $P=QR $ for some $R \in \mathbb{R}[x_1, x_2, \dots, x_n]$ 

 \end{lemma}

Lemma \ref{dl2} follows from Theorem 2 and Lemma 4 in \cite{Mur}.  We outline a proof in the last section for reader's convenience.

We extend the division to a general case and divide  a real analytic function by a harmonic function. 
For the rest of this subsection we suppose that a real analytic function $u$ and a harmonic function $v$ are given.
Consider the Taylor expansions of  $u$ and $v$ at the origin 
\[u =\sum\limits_{i=k}^{\infty} u_i,\quad v =\sum\limits_{i=l}^{\infty} v_i,\] where $u_i$ and $v_i$ denote homogeneous polynomials of degree $i$ and $u_k$ and $v_l$ are non-zero polynomials.
\begin{lemma} \label{l3}
 If $Z(v) \subset Z(u)$, then $Z(v_l) \subset Z(u_k)$.
\end{lemma}

\begin{proof} Assume the contrary: let $y$ be a point such that  $u_k(y) \neq 0$ and $v_l(y)=0$.  We may assume $u_k(y)>0$, so there is an open convex cone $\Gamma$ containing $y$ and $\varepsilon>0$ such that  $u_k(x)>\varepsilon|x|^k$ for any $x\in \Gamma$.  Since $u(x)=u_k(x)+o(|x|^k)$ near the origin, there exists $r>0$ such that for any $ x \in \Gamma $ with $|x|<r$ the inequality $u(x)>0$ holds. 

Clearly, $v_l$ is a harmonic polynomial. By the maximum and minimum principle  there exist $y_+$, $y_-$ arbitrarily close to $y$ with $v_l(y_+)>0$ and $v_l(y_-)<0$,  take $y_+$, $y_-$ within $\Gamma$. Consider $ty_+$ and $ty_-$, where $t$ is a positive real number . If $t$ is small enough, then $v(ty_+)>0$,$v(ty_-)<0$, $|ty_+|<r$ and $|ty_-|<r$. Choose $t$ so that the previous four inequalities hold, then there exists $x$ in the segment connecting $ty_+$ and $ty_-$ such that $v(x)=0$. It is clear that $x \in \Gamma$ and 
$|x|<r$, therefore $u(x)>0$. Thus we obtained a contradiction with $Z(v) \subset Z(u)$.
 \end{proof}

Now, we are in a position to divide an analytic function by a harmonic one as Taylor series.
\begin{lemma} \label{l4}
 If $a\in Z(v) \subset Z(u)$, then there exists a formal power series $f$ such that $u=vf$ as power series centered at $a$.  
\end{lemma}
\begin{proof} We may assume that $a=0$.
 By Lemma \ref{l3}  $Z(v) \subset Z(u) $ implies $ Z(v_l) \subset Z(u_k)$ and by Lemma \ref{dl2}  $u_k$ is divisible by $v_l$. The ratio of $u_k$ and $v_l$  is a homogeneous polynomial of degree $k-l$ which we denote by $f_{k-l}$ and put $\tilde u:=u - v f_{k-l}$. Note that $Z(v) \subset Z(\tilde u)$ and that the degree of the first non-zero polynomial in the Taylor expansion of $\tilde u$ is at least $k+1$. Using similar division step for $\tilde u$ and $v$ (instead of $u$ and $v$), we can find a polynomial $f_{k-l+1}$ such that $\tilde u_{k+1}=f_{k-l+1}v_l$. Further we  put $\tilde{\tilde u}:= \tilde u - v f_{k-l+1}$, then
the degree of the first polynomial in the expansion of $\tilde{\tilde u}$ is at least $k+2$. Applying this division step infinitely many times we obtain a formal equality of power series $u=vf$, where $f= \sum_{j=0}^\infty f_{k-l+j}$.
\end{proof}


\subsection{Estimates of Formal Power Series}
  In the previous subsection we obtained the equality of power series $u=vf$. Next we obtain estimates on the coefficients of $f$ and show that the series converges to  a real analytic function in some neighborhood of the origin. 

We use the usual multi-index notation, $\alpha=(\alpha_1, \alpha_2,...,\alpha_n), \alpha_j\in \Z_+$, $x^\alpha=x_1^{\alpha_1}x_2^{\alpha_2}...\,x_n^{\alpha_n}$; the set of multi-indices is equipped with the partial order  $\alpha \leq \beta$ if $\alpha_i \leq \beta_i$ for each $i$: $1\leq i \leq n$.
	

\begin{lemma} \label{estimate}  Let  $u=\sum\limits_{\alpha} u_\alpha x^{\alpha} $, $v=\sum\limits_{\alpha} v_\alpha x^{\alpha} $, 
$f=\sum\limits_{\alpha} f_\alpha x^{\alpha} $ be formal power series centered at the origin and such that $u=vf$. Suppose that $|v_\alpha| \leq a r^{|\alpha|}$, $|u_\alpha|\le ar^\alpha$ for each $\alpha$ and some positive $a$ and $r$. Assume also $|v_{(k,0,\cdots,0)}|=c>0$ and $v=\sum\limits_{|\alpha| \geq k} v_\alpha $. Then there exist $A$ and $R=(R_1, R_2, \cdots, R_n)$ depending only on $a, c, r, k$, and $n$  such that 
\begin{equation}\label{coef}|f_{\beta}|\leq A R^{\beta} \quad (R^{\beta}:=  R_1^{\beta_1}R_2^{\beta_2}\cdots R_n^{\beta_n})\end{equation} 
for any multi-index $\beta$.  Hence $f$ represents a real analytic function near the origin.
\end{lemma}
 Denote $(k,0,\dots, 0)$ by $\tilde k$.
  By the equality of formal power series $u=vf$ we have
\begin{equation} \label{ineq1}
u_{\beta + \tilde k}= \sum\limits_{\gamma\le \beta+\tilde k, |\gamma|\le|\beta|} f_\gamma v_{\beta+\tilde k - \gamma}  \hbox{ \quad for any multi-index $\beta$}.
\end{equation}

We need an auxiliary proposition which will be used to estimate $|f_\beta|$.
\begin{prop} \label{estimate prop}
 For any $a_0$, $r > 0$  there exist $A=A(a_0,r)$ and $R=(R_1, R_2, \dots , R_n)=R(a_0,r)$ such that for each {multi-index~$\beta$}
  \begin{equation} \label{ineq 2}
a_0r^{|\beta + \tilde k|} + a_0A \sum\limits_{\gamma\le \beta+\tilde k, |\gamma|\le |\beta|,\gamma\neq \beta} R^{\gamma}r^{|\beta+\tilde k -\gamma|} \leq A R^{\beta}.
 \end{equation}

\end{prop}

 We postpone the proof of the Proposition. First, we assume it is true and show that then Lemma \ref{estimate} holds with  $A=A(ac^{-1},r)$ and $R=R(ac^{-1}, r)$. We prove (\ref{coef}) by induction with respect to some lexicographic order on multi-indices. 

 Consider the set of multi-indices $\mathbb{A}:=\{\alpha=(\alpha_1, \dots, \alpha_n): \alpha_i \in \Z_+  \}$ with the order $\prec$ defined by
$$ \gamma \prec \beta \mathop{\iff}\limits^{\hbox{def}} \left [ 
\begin{array}{c c c c c}
\gamma_n< \beta_n \\
\gamma_n= \beta_n, & \gamma_{n-1}< \beta_{n-1} \\
\vdots \\
\gamma_n= \beta_n, & \gamma_{n-1}= \beta_{n-1} & \cdots  & \gamma_2= \beta_2, & \gamma_1< \beta_1.
\end{array}
\right.$$
Then $(\mathbb{A}, \prec)$ is a well-ordered set.
 
\begin{proof}[Proof of Lemma \ref{estimate}]
Denote by $S$ the set of multi-indices $\alpha$ with $|f_{\alpha}|> A R^{\alpha}$. Our goal is to show that $S$ is an empty set. Suppose $S$ is not empty, then $S$ has the least element in the ordering $\prec$, denote it by $\beta$. Let us write $\bigstar$ instead of the following summation condition $
\gamma\le \beta+\tilde k, |\gamma|\le|\beta|, \gamma\neq \beta$; clearly this condition implies $\gamma\prec\beta$. Then (\ref{ineq1}) can be written as
\begin{equation*}
v_{\tilde k} f_\beta = u_{\beta + \tilde k} - \sum\limits_{ \bigstar 
} f_\gamma v_{\beta+\tilde k - \gamma} 
\end{equation*}
 Note that $|f_\gamma|\leq  A R^\gamma$ for any $\gamma \prec \beta$. Keeping in mind that $|v_{\tilde k}|=c>0$ we obtain
\begin{multline*}
|f_\beta| \leq c^{-1} |u_{\beta + \tilde k}| + c^{-1}\sum\limits_{ \bigstar 
} |f_\gamma v_{\beta+\tilde k - \gamma}| \leq\\ c^{-1}a r^{|\beta + \tilde k|} + c^{-1}\sum\limits_{\bigstar}A R^{\gamma} a r^{|\beta+\tilde k -\gamma|} \mathop{\leq}\limits^{\hbox{by (\ref{ineq 2})}} A R^{\beta}.
\end{multline*}
 Therefore $|f_\beta| \leq  A R^{\beta} $ and $\beta \notin S$. Thus $S$ is an empty set and the proof is completed.
\end{proof}

\begin{proof}[Proof of Proposition \ref{estimate prop}] We write $r^\alpha$ for $r^{\sum \alpha_i}$, where $\alpha=(\alpha_1,...,\alpha_n)$ with $a_i\in\Z$.
  Dividing the both sides of (\ref{ineq 2}) by $AR^{\beta}$, we reduce it to the following inequality
\begin{equation} \label{ineq 3}
\frac{a_0r^{k}}{A} \frac{r^{\beta}}{R^{\beta}} + a_0 r^{k} \sum\limits_{\bigstar} \frac{R^{\gamma} r^{\beta - \gamma}}{R^{\beta}} \leq 1.
 \end{equation}
  The first summand can be made less then $1/2$ for all $\beta$ if we choose $A$ and $R=(R_1,...,R_n)$ to be sufficiently large so that 
	\begin{equation*}
\frac{a_0r^{k}}{A}\leq 1/2 \hbox{ and } R_i \geq r  \hbox{ for all } i \in [1,n]. 
 \end{equation*}
 Therefore it  suffices to achieve  \[\sum\limits_{\bigstar} \frac{R^{\gamma} r^{\beta - \gamma}}{R^{\beta}}\leq \frac{1}{2a_0 r^{k}}\]  to make  the inequality  (\ref{ineq 3}) true. By $\bigstar$ we have $\beta_i \geq \gamma_i$ for any $i\in [2,n]$  and $|\beta|\ge |\gamma|$; denote $\beta_i - \gamma_i$ by $\delta_i$ and $|\beta|-|\gamma|$  by  $\delta$.  It's easy to see that  
\begin{equation}\label{sums} \sum\limits_{\bigstar} \frac{R^{\gamma} r^{\beta - \gamma}}{R^{\beta}}\leq \sum\limits_{\bullet}
\left(\frac{R_1}{R_2} \right)^{\delta_2} \left(\frac{R_1}{R_3} \right)^{\delta_3} \dots \left(\frac{R_1}{R_n} \right)^{\delta_n} \left(\frac{r}{R_1}\right)^{\delta},\end{equation} 
where $\bullet$  is the following condition: 
\begin{equation*}
\delta,  \delta_2, \delta_3, \dots, \delta_n \in \Z_+, \quad \delta + \sum\limits_{i\geq 2} \delta_i >0. \end{equation*}
  Note that the right hand side of (\ref{sums})  is the product of geometric progressions without the first term $1$. Therefore \begin{multline*}\sum\limits_{\bullet}
\left(\frac{R_1}{R_2} \right)^{\delta_2} \left(\frac{R_1}{R_3} \right)^{\delta_3} \dots \left(\frac{R_1}{R_n} \right)^{\delta_n} \left(\frac{r}{R_1}\right)^{\delta}=\\
\frac{1}{1-\frac{R_1}{R_2}} \frac{1}{1-\frac{R_1}{R_3}}\dots \frac{1}{1-\frac{R_1}{R_n}} \frac{1}{1-\frac{r}{R_1}} - 1.\end{multline*}
And the last expression can be made arbitrarily small by a proper choice of $R$ (we can take $r << R_1, R_1 << R_2=R_3=..=R_n$).
\end{proof}


\subsection{Division by a harmonic function}
  Lemma \ref{estimate} together with Lemma \ref{l4} gives us the following theorem.
\begin{theorem} \label{t1}
Suppose $u$ is a real-analytic function and $v$ is a harmonic function, both functions are defined in some domain $\Omega\subset \mathbb{R}^n$, $n\geq 2$. If $Z(v) \subset Z(u)$, then there exist a real-analytic function $f$ in $\Omega$ such that $u=vf$. 
\end{theorem}\label{th:ra}
\begin{proof}
Indeed, for any $x_0 \in \Omega $ we know that the Taylor series at $x_0$ of $v$ is a divisor of a Taylor series at $x_0$ of $u$. The only obstacle is to show that formal power series $f=u/v$ centered at $x_0$  has a positive radius of convergence; and here Lemma \ref{estimate} comes into play. Represent $v$ as a sum of monomials: $v=\sum v_\alpha (x-x_0)^{\alpha}$.  If $v$ is not identically zero, we can take $k\ge 0$ such that $v_\alpha=0$  for any multi-index $\alpha$ with $|\alpha|<k$ and there is $\alpha$ with $|\alpha|=k$: $v_{\alpha} \neq 0$. Further, we can rotate the coordinate axes to obtain $v_{(k,0,\dots,0)}\neq 0$ and apply Lemma \ref{estimate}. Finally the estimate (\ref{coef}) implies the  absolute convergence of the power series of $f$ in some neighborhood of $x_0$.
\end{proof}


\subsection{Maximum and Minimum Principle for ratios of harmonic functions.}  
Let $u$ and $v$ be harmonic  in $\Omega$ and such that $Z(u)\supset Z(v)$.  We already know that there exists a real analytic in $\Omega$ function $f$ such that $u=fv$. 

\begin{theorem} Let  $f$ be as above. Then $f$ enjoys the maximum and minimum principle$,$ i.e. for any subdomain $O\Subset \Omega$
 $$\max \limits_{\partial O} f =  \max\limits_{ \bar O} f  \quad \& \quad 
 \min \limits_{\partial O} f =  \min\limits_{ \bar O} f.$$
\end{theorem}
\begin{proof}  Let $M=\max_{\partial O} f$ and $m=\min_{\partial O} f$. 
 Let $D \subset \Omega$ be any nodal domain  of $v$ that intersects $O$. Let  $\Gamma_0$ denote  $O \cap \partial D$  and  $\Gamma_1$ denote $\bar{D} \cap \partial O$.
 We may assume  $v$ is positive in $D$. It's clear that $mv\le u \leq M v $ on ${ \Gamma_1}$ and surely $mv\le u \leq M v  $ on $\Gamma_0$ because $u$ and $v$ vanish on $\Gamma_0$. Therefore $mv\le u \leq M v$ on $D\cap O$  by  the standard maximum principle. Hence we have $m\leq f \leq M$ everywhere in $\bar O$. Thus $\max \limits_{\bar O} f =  \max\limits_{ \partial O} f$ and $\min \limits_{ \bar O } f =  \min\limits_{ \partial O} f$.
\end{proof}

 \begin{remark}
The maximum principle for ratios of harmonic functions  is strict, i.e. a local maximum or minimum can not be attained at an interior point unless $f$ is a constant function.  In order to prove it one can show that if $f=0$ at some interior point, then $f$ changes a sign. The last claim can be proved with the help of Proposition \ref{pr1}.
\end{remark}



\section{Harnack Inequality for the ratios of harmonic  in $\R^3$ and Gradient Estimate}

In this section we prove Theorem \ref{th:H} and then deduce Theorem \ref{th:full}. We fix a harmonic function $w$ in a subdomain $\Omega$ of $\R^3$. 

\subsection{Structure of the nodal set of harmonic function in dimension three}
  Let $Z=Z(w)\subset\Omega$.  We say that a point $x\in Z$ is good if for each nodal domain $\Omega_i$ with $\partial \Omega_i \ni x$ the following holds: there exists a neighborhood $W$ of $x$ such that $\partial \Omega_i \cap W$ can be parametrized by a graph of a Lipschitz function, i.e. $\partial \Omega_i$ is Lipschitz in some neighborhood of $x$.    We say that a point $x \in Z$ is a bad point if it is not good. 

We have $Z=Z_0\cup Z_1$, where $Z_0=\{x: w(x)=0, \nabla w(x)\neq0\}$ and $Z_1=\{x: w(x)=0, \nabla w(x)=0\}$. In some neighborhood of each point of $Z_0$ the nodal set is a smooth surface and all points of $Z_0$ are good; $Z_1$ is the critical set of $w$, it is locally a finite union of analytic curves and a discrete set of points. We refer here to a general structure theorem for real analytic varieties  of \L ojasiewicz, see \cite{L} or \cite[Chapter 6.3]{KP}.  Consider any analytic curve $\Gamma$ in $Z_1$ and for each $x\in Z$ denote by $d(x)$ the depth of zero at $x$, i.e. $d(x)$ is the degree of the first non-zero  homogeneous polynomial in the Taylor series of $w$ at $x$. Suppose there is a sequence of points $\{x_i\}_{i=1}^{\infty}$ on $\Gamma$ converging to an interior point $x_\infty$ of $\Gamma$ such that $d(x_i) \geq k$ for some $k \in \mathbb{N}$, then the real analyticity of $w$ and $\Gamma$ implies $d(x)\geq k $ for any $x\in \Gamma$. Hence there exists $k \in \mathbb{N}$ such that $d(x)=k$ for all $x \in \Gamma$ except for at most a countable set of points on $\Gamma$ with at most two accumulation points at the ends of  the curve (see also the proof of Lemma 2.4 in \cite{HHN} for a similar decomposition of the critical set).
\begin{lemma}\label{l:good}
Let  $x$ be an interior point of $\Gamma$ and let  $U$ be a neighborhood of $x$ such that for any $y \in \Gamma \cap U$:
d(y)=k. Then $x$ is a good point. 
\end{lemma}

The main idea of the proof is to consider the first non-zero term of the Taylor expansion of  $w$ at each point $y\in \Gamma\cap U$, it is a homogeneous harmonic polynomial of degree $k$ of two variables in the plane orthogonal to $\Gamma$ at $y$ and the gradient of this term restricted on the plane is bounded from below by the $(k-1)$st power of the distance to $y$.

\begin{proof} We assume without loss of generality that $x=0$.
Let $I$ be an interval containing zero and let $\Gamma:I\rightarrow\Omega$ be a parametrization  such that  $\Gamma(0)=0, \Gamma(t)=(x(t), y(t),t)$,  $\Gamma'(0)=(0,0,1)$. Assume further that $d(\Gamma(t))=k$ for  each $t\in(-r,r)$. Let $p_t(x,y,z)$ be the $k$th Taylor polynomial of $w$ at the point $\Gamma(t)$,
\[w(x,y,z)=p_t(x-x(t), y-y(t), z-t)+q_t(x-x(t),y-y(t),z-t),\] where $q_t$ is the remainder in the Taylor expansion. By the classical estimates of the derivatives of harmonic functions we have $|q_t(X)|\le C|X|^{k+1}$ and $|\nabla q_t(X)|\le C|X|^{k}$ for $|X|$ small enough uniformly in $t$, here $X$ denotes $(x,y,z)$. Clearly, $p_t$ is a homogeneous harmonic polynomial of degree $k$ whose coefficients are real analytic in $t$. 

Fix some point $t_0\in (-r,r)$ and let $v_0=\nabla\Gamma(t_0)=(x'(t_0), y'(t_0),1)$ be the tangent vector to $\Gamma$ at $\Gamma(t_0)$. Let $f$ be some partial derivative of $w$ of order $k-1$, $f=\partial^\alpha w$, $|\alpha|=k-1$. Then $f(\Gamma(t))=0$ when $-r<t<r$  and therefore $\langle\nabla f(\Gamma(t_0)), v_0\rangle=0$. On the other hand  $(\nabla f)(\Gamma(t_0))=\nabla (\partial^\alpha p_{t_0})(0)$
and all partial derivatives of $p_{t_0}$ of order $k$ are constants.
Hence $\langle\nabla \partial^\alpha p_{t_0}, v_0\rangle=0$. We know that $p_{t_0}(\xi)$ is a homogeneous polynomial of order $k$. Then $\langle \nabla p(\xi), v_0\rangle$ is a homogeneous polynomial of order $k-1$ and it can be written as $\sum_{|\alpha|=k-1}c_\alpha\xi^\alpha$. The coefficient $c_\alpha$ of that polynomial is equal to $(\alpha!)^{-1}\partial^{\alpha}\langle \nabla p_{t_0}, v_0\rangle=0$.
Therefore $p(\xi)$ does not depend on $\langle\xi , v_0\rangle$ (i.e.  $p_{t_0}(\xi + s v_0)=p_{t_0}(\xi)$ for any $s\in \mathbb{R}$). In other words, $p(\xi)$ is actually a homogeneous harmonic polynomial of degree $k$ of two variables in appropriate coordinates. Then there exists an orthogonal basis $\{a(t_0), b(t_0), v(t_0)\}$ such that $p(t_0)(X)=c(t_0)\Re\left\{\left(\langle X,a(t_0)\rangle+i\langle X, b(t_0)\rangle\right)^k\right\}$, $c(t_0)\neq 0$.

We may choose $a(t), b(t), c(t)$ to be real analytic functions in $t$, when $|t|$ is small enough,  with $|c(t)|>c_0>0$.
We denote $v(t)=\nabla \Gamma (t)=(x'(t), y'(t),1)$ and remark that $|v(t)|<1+\delta$ for small $|t|$. The projections of $a(t)$ and $b(t)$ onto the plane $\{z=0\}$ are denoted by $a_1(t)$ and $ b_1(t)$ respectively. We will also need the matrix $A(t)\in M_2$ which is the inverse of the matrix $[a_1(t), b_1(t)]$, it exists when $t$ is small enough and depends analytically on $t$.

Our aim is to show that each nodal domain of $w$ near the origin is a Lipschitz domain. We will perform some diffeomorphic changes of  variables to simplify the geometry of the  nodal set near zero. From this point we don't use the fact that the function $w$ is harmonic.

First, let us consider the map $F(x,y,z)=(x+x(z), y+y(z), z)$ defined on some neighborhood $U\subset\R^2\times[-r,r]$ of the origin; clearly it is a diffeomorphism. We define $w_1=w\circ F$.  Then $w_1$ vanishes on the $z$-axis with all its derivatives  up to order $k-1$. Let $X=(x,y,z)$, we have 
\begin{multline*}
w_1(x,y,z+t)=p_t(x+x(z+t)-x(t), y+y(z+t)-y(t),z)+O(|X|^{k+1})=\\
p_t(x+x'(t)z, y+y'(t)z, z)+ O(|X|^{k+1}),\ |X|\rightarrow 0.
\end{multline*}
Further, we have \begin{multline*}
p_t(x+x'(t)z, y+y'(t)z,z)=\\c(t)\Re\left\{\left(\langle(x,y,0)+zv(t),a(t)\rangle+\langle (x,y,0)+zv(t), b(t)\rangle\right)^k\right\}.
\end{multline*} 
Taking into account that $a(t)$ and $b(t)$ are orthogonal to $v(t)$, we conclude that $w_1(x,y,z+t)=c(t)\Re\{((X, a_1(t))+i(X,b_1(t)))^k\}+Q_t(X)$, where $|Q_t(X)|\le C|X|^{k+1}$ and $|\nabla Q_t(X)|\le C|X|^{k}$.

Next let $G(x,y,z)=(A(z)(x,y),z)$; it is a diffeomorphism in a neighborhood of the origin,  we consider $w_2(x,y,z)=c^{-1}(z)(w_1\circ G)(x,y,z)$. Then $w_2$ vanishes on $(0,0,z)$ for small $z$ with all its derivatives of order up to $k-1$ and $k$th Taylor polynomial of $w_2$ at each point $(0,0,z)$ is $\Re(x+iy)^k$. It suffices to show that the nodal domains of $w_2$ are Lipschitz near the origin. 

Let us fix $z_0$ and consider the plane $(x,y,z_0)$; the restriction of $w_2$ to this plane has the form $\Re(x+iy)^k+\tilde{q}_{z_0}(x,y)$ where the remainder satisfies $|\nabla \tilde{q}_{z_0}(x,y)|\le C(x^2+y^2)^{k/2}$, while the gradient of the main term is greater than or equal to $c(x^2+y^2)^{(k-1)/2}$.  It means that the gradient of $w_2$ does not vanish in $(B\setminus\{0\})\times(-r_1,r_1)$, where $B$ is a small enough two-dimensional ball around the origin.  

Let $D_0$ be the nodal domain of $w_2$ in $B_{2s}$  that contains  the point $(s,0,0)$ for $s>0$ is small enough.   
 Take any $\phi_1, \phi_2$ with $0<\phi_1<\frac{\pi}{2k}<\phi_2<\frac{\pi} {k}$. We consider the domain \[U=\{(x,y,z):\ |z|< r_0,\ 0<x<x_0,\ x \tan \phi_1 <|y|< x\tan \phi_2 \},\]
 which consists of two connected components. We want to show that $U\cup\{z=0\}$ contains $\partial D_0\cap V$ for some neighborhood $V$ of the origin. Note that for $x_0,r_0>0$ small enough, we have  
\begin{gather*}
w_2(x, x \tan \phi , z)= x^k(\cos \phi)^{-k}\cos k\phi+ O(|x|^{k+1})>0, \ |\phi|\le\phi_1,\\
w_2(x, \pm x \tan \phi_2 , z)= x^k(\cos \phi_2)^{-k}\cos k\phi_2+ O(|x|^{k+1})<0, \ \quad {\text{and}}\\  
 \partial_x w_2(x,y,z)\ge c(|x|^2+|y|^2)^{(k-1)/2}>0\ \text{ in } \Omega.
\end{gather*} 
 Then $(\partial D_0 \cap B_{\varepsilon}) \setminus (\{0,0,z\})\subset \Omega$, where $\varepsilon > 0$ is small enough. Further, $\partial D_0$ is a graph over the plane $(0,y,z)$, and by the implicit function theorem if $\partial D_0$ is given by $(g(y,z),y, z)$ then 
\[|\nabla g(y,z)|\le|\nabla w_2(g(y,z),y,z)|/|\partial_x w_2(g(y,z), y,z)|\le C,\quad  {\text{when}}\quad  y\neq 0.\]
Then $g(y,z)$ is a continuous function differentiable everywhere except for the line $\{y=0\}$ with uniformly bounded derivative, hence it is Lipschitz  in $B_\varepsilon$.
 The argument above shows that there are exactly $2k$  nodal domains in $B_\varepsilon$ and each of them can be represented by a graph of a Lipschitz function.
\end{proof}

The proof above suggests that  there exist a neighborhood $V$ the origin and a diffeomorphism $H:V\rightarrow B$ such that $w(x,y,z)=g_k(H(x,y,z))$, where $g_k(x,y,z)=\Re(x+iy)^k$. We conjecture that it can be constructed like in the proof of Kuiper-Kuo theorem in \cite{C}, the difference is that one needs to apply it to an analytic one-parameter family of functions; however we were not able to find such a construction in the literature.

\subsection{Proof of Theorem \ref{th:H}}
 
First, Lemma \ref{l:good} implies 
\begin{corollary} \label{count}
The set of bad points in $Z\cap V$ is at most countable  set with a finite number of accumulation points.
\end{corollary}
 This corollary will be used in the proof of the following theorem.
\begin{theorem}\label{th:Ha}
 Let  $u$ and $v$ be any harmonic in $\Omega$ functions with $Z(u)=Z(v)=Z$ let $x_0$ be a point in $Z$.  Let $f$ be the ratio of $u$ and $v$.
 There exists a positive constant $C=C(Z,x_0)$ and a ball $B_r(x_0)\subset \Omega$ with center $x_0$ and some radius $r$ such that
$$ \inf\limits_{ B_r(x_0)} |f| \geq  C\sup\limits_{ B_r(x_0)} |f| $$
\end{theorem}

\begin{proof} We already know from Theorem \ref{t1} that $f$ is a continuous function in $\Omega$. 
Corollary \ref{count} implies that there exists a spherical layer $Q:=B_{R}(x_0)\setminus \overline{B_{r}}(x_0)$ with $R> r \geq 0$ and $\overline{B_{R}}(x_0) \subset \Omega $ such that every $x \in Q\cap Z$ is a good point. Consider the sphere $S$ of radius $\frac{r+R}{2}$ with center at $x_0$.   Let $\Omega_i$ be any nodal domain with non-empty intersection with $S$ and let $S_i$ denote $S\cap \overline{\Omega_i}$. Note that $S_i$ is compact subset of $Q$. By the boundary Harnack principle for Lipschitz domains (we refer the reader to \cite{A}, \cite[Chapter 8.7]{AG},  and the references therein), there exists a constant $C_i$ such that $\max\limits_{S_i} |f| \leq C_i \min\limits_{S_i} |f|$. Put $C:= \prod \limits_{i} C_i$. It can be easily checked by induction on the number of nodal domains in $Q$ that
 $$\max\limits_{S} |f| \leq C \min\limits_{S} |f|.$$
  By the maximum and minimum principle for harmonic fractions we have $\sup_{B_r} |f| \leq \max_S |f|$ and $ \inf\limits_S |f| \leq \inf\limits_{B_r}|f|$. Thus  \[\sup\limits_{ B_r(x_0)} |f|\leq C  \inf\limits_{ B_r(x_0)} |f| .\]
\end{proof}
   Now, Theorem \ref{th:H} follows from the previous theorem and standard compactness arguments.


\subsection{Proof of Theorem \ref{th:full}}
Let $Z$ be a zero set of some harmonic function $w$ in $\Omega\subset \R^3$. 
Let $x_0\in\Omega\setminus Z$ and define \[\F_0(Z)=\{u:\Omega\rightarrow\R: \Delta u=0, Z(u)=Z, u(x_0)=w(x_0)\}.\] 
Clearly for any $u$ with the nodal set  $Z(u)=Z$ there exists a constant $c_u$ such that $c_uu\in \F_0(Z)$.
\begin{lemma} \label{l5}
   Consider a point $y \in \Omega_0$. There exist a neighborhood $V_y$ of $y$, $V_y\subset \Omega_0$, and positive constants $A_y$, $R_y$ such that for any $x\in V_y$ the inequality 
	\[\frac{|D^{\alpha} (f)|}{\alpha !}(x)\leq A_y R_y^{|\alpha|}\] holds whenever $f= u/v$ for some $u,v\in \F_0(Z)$.
\end{lemma}
\begin{proof} To simplify the notation let $y=0$.
 By Theorem \ref{th:H} there exist a constant $C=C(y,w)$ and  a neighborhood  $V$ of $0$ such that $\frac{1}{C} \leq |\frac{u}{w}| (x) \leq C$ and  $\frac{1}{C} \leq |\frac{v}{w}| (x) \leq C$ for any $x \in  V$.  Let $M=\sup_V|w|$, then $|u|$ and $|v|$ are bounded by $CM$ in $V$.  By the standard Cauchy estimates, there exist positive numbers $a=a(y,w)$ and $r=r(y,w)$ such that
\begin{equation}
|u_\alpha|\leq a r^{|\alpha|} \hbox{ and } |v_\alpha|\leq a r^{|\alpha|}.
\end{equation}
Let $w=\sum\limits_{i=k}^\infty w_i$ be the decomposition of $w$ into the sum of homogeneous harmonic polynomials and let $w_k$ be the non-zero polynomial of the least degree $k$. Let us rotate the coordinate lines to make $w_{(k,0,\dots , 0)} \neq 0 $. Let further $u=\sum\limits_{i=l}^\infty u_i$ and $v=\sum\limits_{i=m}^\infty v_i$ be analogous sums for $u$ and $v$.  By Lemma \ref{l3} we have $Z(w_k)=Z(u_l)=Z(v_m)$, then Lemma \ref{dl2} implies $k=l=m$ and $u_k= c_1 w_k$, $v_k= c_2 w_k$, where $c_1$, $c_2$ are non-zero constants. 
By l'H\^{o}pital's rule $\frac{v_{(k,0,\dots , 0)}}{w_{(k,0,\dots , 0)}}=\frac{v}{w}(0)$. Since $\frac{1}{C} \leq |\frac{v}{w}| (x) \leq C$, we have $|v_{(k,0,\dots , 0)}| \geq C^{-1}|w_{(k,0,...,0)}|$.  Now we are in position to apply  Lemma \ref{estimate}, the constants $A_y, R_y$ depend on $y, w$ but does not depend on $u$ and $v$.

\end{proof}

 Now our main result  is a straightforward consequence of the previous lemma and a standard compactness argument.
 
\begin{proof}[Proof of Theorem \ref{th:full}]
 Real analyticity of $f$ was proved in Theorem \ref{t1}.
 Lemma \ref{l4} claims that for any $y$ there exist $A_y$ and $R_y$ such that $\frac{|D^\alpha f(y)|}{\alpha !} \leq A_y R_y^{\alpha}$ for any multi-index $\alpha$. Since $f$ is real-analytic, then there exist a neighborhood of $y$ denoted by $V_y $ such that    $\frac{|D^\alpha f(x)|}{\alpha !} \leq \tilde A_y (\tilde R_y)^{\alpha}$  for any $x\in V_y$. Note that $K \subset \bigcup\limits_{y\in K} V_y $. Since $K$ is a compact set, there exist a finite set $\{y_1,y_2 \dots, y_m\}$ such that $K \subset \bigcup\limits_{i \in \{1, \dots, m\} } V_{y_i} $. Take $A:= \max \limits_{i} \{\tilde A_{y_i}\}$ and $R:=\max \limits_{i} \{\tilde R_{y_i}\}$.
\end{proof}

\if false

\subsection{Why does it work in $\mathbb{R}^3 $?}

 The goal of this section is to obtain the Harnack inequality for harmonic fractions in $\mathbb{R}^3$. In dimension two there is a natural and successful way - to apply BHP. It is possible due to the fact that a nodal domain in $\mathbb{R}^2$ is locally a graph of a Lipschitz function. But we can not apply BHP directly in $\mathbb{R}^3$ like it was possible in $\mathbb{R}^2$. A nodal domain in $\mathbb{R}^3$ is not necessarily Lipschitz, it might have a complicated topological structure, may not be represented locally as a graph of a function and for the last reason it can not be regarded as a Holder domain too. We do not expext a nodal domain in $\mathbb{R}^3$ to necesserily enjoy the Harnack chain condition (see example ... ). Nevertheless, it is possible to combine BHP and MP and obtain a proof. 
Unfortunately, the proof is purely 3-dimensional and we were not able to extend it to higher dimensions.
\begin{proof}
  Here will be proof.
\end{proof}
\fi 

\section{Concluding remarks}
\subsection{Nodal sets of harmonic functions and harmonic polynomials}
It is an interesting question to which extend the nodal set of a harmonic function (or more generally of a solution to some elliptic equation)  resembles the nodal set of its first non-zero homogeneous polynomial. We refer the reader to \cite{HHN} and references therein. In dimension two the nodal set of harmonic functions and solutions to elliptic equations locally look like regular intersections of curves. In higher dimensions we implicitly used some information on the nodal sets to divide harmonic functions sharing the same zeros in $\R^n$ and prove that most of the points of the nodal set in dimension three are good. However the following example shows that starting from dimension three the nodal sets may have complicated local geometry.

\begin{example}\label{e:NH} Let $H(x,y,z)=x^2-y^2+z^3-3x^2z$, clearly it is a harmonic polynomial. The intersection of its nodal set with a plane $\{0,0,z\}$ is the union of two orthogonal lines when $z=0$ and of two hyperbolas for $z\neq 0$. There are only two nodal domains $\Omega_1$ and $\Omega_2$ (not four like for the case of $x^2-y^2$) and those nodal domains are not Lipschitz. Moreover the Harnack chain condition does not hold for $\Omega_{1,2}$  (see \cite{A, JK} for the definition). We don't know if the boundary Harnack principle is valid for $\Omega_{1,2}$.
\end{example}

\subsection{Differential equation for the ratio}
One can think about the ratio $f$ as a positive solution of the following second order degenerate  elliptic equation
\[
\dv(v^2\nabla f)=0.\]
Unfortunately the coefficient is very singular, $v^2$ does not belong to the Muckenhoupt class $A_2$ when $v$ changes sign, and we are not able to apply the Harnack inequality for degenerate elliptic operators with $A_2$ condition on weight (see \cite{FKS})  here. It would be interesting to see if one can use harmonicity of $v$ to obtain Harnack inequality for positive solutions of such equations in $\R^n$. A more delicate equation for the $\log f$ was used in \cite{M} in dimension 2. 

Another interesting question is when the equation above admits any non-trivial positive solutions and how large this family may be.

\subsection{Zeroes and Division of Real-valued Polynomials in Several Variables}
We suggest a proof of Lemma \ref{dl2} in this subsection.
The following division follows from general results in algebraic geometry,  we borrowed it from \cite[Chapter 5]{MOV}.

\begin{lemma*}[Division Lemma, Mateu, Orobitg, Verdera]\label{dl1}
Let $Q$ and $P$  be polynomials in $\mathbb{R}[x_1, x_2, \dots, x_n]$. Suppose that $H^{n-1}(Z(P)\cap Z(Q))>0$ and $Q$ is  irreducible. Then there is $R \in \mathbb{R}[x_1, x_2, \dots, x_n]$ such that
$P=QR $.
\end{lemma*}

 We are going to replace irreducible polynomial  $Q$  by a homogeneous harmonic one to prove Lemma \ref{dl2}.

 We write $S \sqsubset T$  in case sets $S$, $T \subset \mathbb{R}^{n}$ satisfy $H^{n-1}(S \setminus T)=0$. 
  Lemma \ref{dl2} follows from two propositions below.

\begin{prop} \label{pr1}
If $Q$ is a non-zero homogeneous harmonic polynomial and $Q_1$ is a non-constant divisor of $Q$, then $Q_1$ changes sign  and $H^{n-1}(Z(Q_1))>0$.
\end{prop}

\begin{prop} \label{pr2}
 Suppose polynomials $P$ and $Q$ enjoy the following properties:

\begin{enumerate}
\item \label{pr21} 
$Z(Q)  \sqsubset Z(P)$,
\item  \label{pr22}
If $Q_1$ is a non-constant divisor of $Q$, then $Q_1$ changes sign.
\end{enumerate}

Then $P$ is divisible by $Q$.
\end{prop}

 \begin{proof}[Proof of Proposition \ref{pr1}] 
If $Q_1$ changes sign, then $H^{n-1}(Z(Q_1))>0$ (see also the dimension lemma in \cite[Chapter 5]{MOV}). We may therefore assume $Q_1 \geq 0$ and try to obtain a contradiction.  Let $Q=Q_1Q_2$ then clearly, $Q_2 Q =Q_1Q_2^2\geq 0$.  The degree of $Q_2$ is strictly less then the degree of $Q$. Since a homogeneous harmonic polynomial is orthogonal on sphere to any polynomial of smaller degree,  $ \int\limits_{S_r} Q_2(x) Q(x) d\sigma(x)= 0,$ where $S_r$ is the $(n-1)$-dimensional sphere with center $0$ and some radius $r$, $\sigma$ is the surface Lebesgue measure. Keeping in mind that $Q_2 Q \geq 0$ we obtain $Q_2 Q = 0$ a.e. on $S_r$. Since $r$ is an arbitrary positive number, $Q_2 Q \equiv 0$. We therefore have $Q_1 \equiv 0$ and a contradiction is obtained ($Q_1$ is a non-constant polynomial).
\end{proof}

 \begin{proof}[Proof of Proposition \ref{pr2}]
 If $P$ or $Q$ is a constant function, then the statement is trivial. We argue by induction on the degree of $Q$. Consider any irreducible non-constant divisor of $Q$ and denote it by $Q_1$. We know that $Z(Q_1) \subset Z(Q) \sqsubset Z(P)$ and that $H^{n-1} (Z(Q_1)) > 0$, hence  $H^{n-1}(Z(P)\cap Z(Q_1))>0$. Applying Division  Lemma (see above) to $P$ and $Q_1$, we see that $P$ is divisible by $Q_1$. Put $ \tilde P:=P/Q_1$ and $\tilde Q:= Q/Q_1$. It's clear that $\tilde Q$ enjoys the property (\ref{pr22}).

 Now, we want to show that $Z(\tilde Q)\sqsubset Z( \tilde P)$. Assume it is not true, i.e.,  $H^{n-1}(Z(\tilde Q)\setminus Z(\tilde P))>0$. Clearly $Z(P)=Z(\tilde P)\cup Z(Q_1)$,  and by the property (\ref{pr21}), $H^{n-1}(Z( Q)\setminus Z(  P))=0$.  Hence $H^{n-1}(Z(Q\setminus Q_1)\cap Z(  Q_1))>0$. Then by Lemma \ref{dl1}, $Q_1|(Q/Q_1)$ and $Q_1^2|Q$, which contradicts to the property (\ref{pr22}). 

 We see that $\tilde P$ and $\tilde Q$ enjoy the properties (\ref{pr21}) and (\ref{pr22}), since the degree of $\tilde Q$ is less than the degree of $Q$ we obtain $\tilde P=\tilde Q R$ and then $P=QR$.
\end{proof}
\begin{remark}
 A mind-reader might see that in Theorem \ref{t1} we can replace $Z(v) \subset Z(u)$ by $Z(v) \sqsubset Z(u)$:

 Suppose $u$ is a real-analytic function and $v$ is a harmonic function, both functions are defined in some domain $\Omega \subset \mathbb{R}^n$, $n\geq 2$. If $Z(v) \sqsubset Z(u)$, then there exist a real-analytic function $f$ in $\Omega$ such that $u=vf$.

\end{remark}
  Actually  $Z(v) \sqsubset Z(u)$ implies $Z(v) \subset Z(u)$ if $v$ is harmonic and $u$ is real analytic.
 The previous remark was done on purpose: 
 if we apply it several times we can obtain the following theorem.
\begin{theorem}
 Suppose $u$ is a real-analytic function and $v_1, v_2, \dots , v_k $ are harmonic functions, all functions are defined in some domain $\Omega \subset \mathbb{R}^n$, $n\geq 2$. If $Z(v_i) \subset Z(u)$ for any $i \in [1,k]$ and $H^{n-1}(Z(v_i)\cap Z(v_j))=0$ for $i \neq j$, then there exist a real-analytic function $f$ in $\Omega$ such that $u= f\prod\limits_{i=1}^{k}v_i $.
\end{theorem}
\subsection{Real and Complex zeros of harmonic functions} Our results show that if harmonic functions $u$ and $v$ have the same zero set $Z$ in a ball $B\subset\R^n$, then there zero sets in $\C^n$  coincide at least at some complex neighborhood of a smaller real ball $b$. Moreover, if a real analytic function vanishes on the zero set of a harmonic function, then its complex zero set contains all complex zeroes of the harmonic function in some complex neighborhood.    If $n=2$ or $3$ Theorem \ref{th:full} implies that  this neighborhood can be chosen to depend on $Z$ only and not on $u$ and $v$, i.e. the real zeros of a harmonic function uniquely determine  its complex zeros in some complex neighborhood. It would be interesting to prove this directly and see if this neighborhood can be chosen to depend on $Z$ only in higher dimensions as well.  
\subsection{Compactness conjecture for harmonic functions with a fixed zero set.}

  It is a classical fact that a family of positive and harmonic in $B_1$ functions with value $1$ at $0$ is a normal family.
     Let $Z$ be any subset of $B_1$. Consider the set $F_Z$ of all harmonic in $B_1$ functions $u$ such that $Z(u)=Z$ and $u(0)=1$. We conjecture that  in any dimension $F_Z$ is a normal family. For the dimension $n=2,3$ it follows from the Harnack inequality for ratios of harmonic functions.


\begin{thebibliography}{ZZZZ}
\bibitem{A} Hiroaki Aikawa, Potential analysis on nonsmooth domains, Martin boundary and boundary Harnack principle, in Complex Analysis and Potential Theory, Vol. 55 of CRM Proc. Lecture Notes, Amer. Math. Soc., Providence, RI (2012), 235--253.

\bibitem{AG} David H. Armitage and Stephen J. Gardiner, Classical Potential Theory, Springer, 2001.

\bibitem{BC} Marcel~Brelot and Gustave~Choquet, Polynomes harmoniques et polyharmoniques, Colloque sur les Equations aux Derivees Partielles, Brussels, 1954, 45--46.

\bibitem{C} Chuan I. Chu, On the Kuiper-Kuo theorem, Canad. Math. Bull., 34 (1991), 175--180.

\bibitem{CFZ} Michael  Cranston, Eugene  Fabes, and Zhong Zhao, Conditional gauge and potential theory for the Schr\"odinger
operator, Trans. Amer. Math. Soc., 307(1) (1988), 171--194.

\bibitem{JK} David S. Jerison and Carlos E. Kenig, Boundary behavior of harmonic functions in non-tangentially accessible domains, Advances in Mathematics, 46 (1982), 80--147.

\bibitem{FKS} Eugene B. Fabes, Carlos E. Kenig, and  Raul P. Serapioni,
The local regularity of solutions of degenerate elliptic equations. 
Comm. Partial Differential Equations 7 (1982), no. 1, 77--116.  

\bibitem{HHN} Maria Hoffmann-Ostenhof, Thomas Hoffmann-Ostenhof, and Nikolai Nadirashvili, Critical sets of smooth solutions to elliptic equations in dimension 3, Indiana Univ. Math. J., 45 (1996), no. 1, 15--37.

\bibitem{KP} Steven G. Kranz and Harold R. Parks, A primer of Real Analytic Functions, Second ed., Birkh\"auser Verlag, 2002.

\bibitem{L} Stanis\l aw \L ojasiewicz, Sur le probl\`eme de la division. (French), Studia Math. 18 (1959), 87--136.

\bibitem{M} Dan Mangoubi, A gradient estimate for harmonic functions sharing the same zero set, Electron. Res. Announc. Math. Sci. 21 (2014), to appear.

\bibitem{MOV} Joan Mateu, Joan Orobitg, and Joan Verdera, Estimates for the maximal singular integral in terms of the singular integral: the case of even kernels, Annals of mathematics, 174 (2011), 1429--1483.

\bibitem{Mur} B. H. Murdoch, A theorem on harmonic functios, J. London Math. Soc. 39 (1964), 581--588.
\end{thebibliography}
\end{document}